\documentclass[a4paper]{article}

\usepackage{graphicx}
\usepackage[utf8]{inputenc}
\usepackage{amsmath, amsthm, amssymb, amsfonts,mathtools}

\usepackage{fourier}
\usepackage[french,english]{babel}
\usepackage{microtype}
\usepackage[backend=bibtex,style=alphabetic]{biblatex}

\newtheorem{theorem}{Theorem}[section]
\newtheorem{lemma}[theorem]{Lemma}
\newtheorem{proposition}[theorem]{Proposition}
\newtheorem{corollary}[theorem]{Corollary}
\newtheorem{definition}[theorem]{Definition}

\newtheorem*{proposition*}{Proposition}
\newtheorem*{lemma*}{Lemma}
\newtheorem*{theorem*}{Theorem}
\newtheorem*{corollary*}{Corollary}

\title{Closed ray affine manifolds}

\author{Raphaël V.~{\sc Alexandre}\footnote{Institut de Math\'ematiques de Jussieu-Paris Rive Gauche, Sorbonne Université, 4 place Jussieu, 75252 Paris Cédex, France.
ACG,
OURAGAN (IMJ-PRG, INRIA Paris, Sorbonne Université, Université de Paris, CNRS).
Email address: {\tt raphael.alexandre@math.cnrs.fr}.}}

\DeclareMathOperator*\id{id}
\newcommand\R{\mathbf{R}}

\DeclareMathOperator*\Aff{Aff}

\DeclareMathOperator*\GL{GL}

\DeclareMathOperator*\Aut{Aut}

\newcommand{\iI}{\mathopen{[}0\,,1\mathclose{]}}
\newcommand\rT{{\rm T}}

\newcommand\thmpartialcomp{
Let $(G_1,\R^n)$ be a rank one ray geometry.  Let $M$ be
a closed $(G_1,\R^n)$-manifold.
Then $M$ is either complete or there exists an affine subspace $I\subset \R^n$ such that $D\colon\widetilde M \to \R^n-I$ is a cover onto its image.
}
\newcommand\thmauto{
Let $(G_1,\R^n)$ be a rank one  ray geometry.  Let $M$ be
a closed $(G_1,\R^n)$-manifold. If $\Aut(M)$ does not act properly on $M$ then $M$ is complete.
}

\newcommand\thmmarkus{
Let $(G_1,\R^n)$ be a rank one ray geometry with parallel volume. Every closed $(G_1,\R^n)$-manifold   is complete.
}

\bibliography{ref.bib}

\begin{document}
\maketitle

\begin{abstract}
We consider  closed manifolds that possess a so called  rank one \emph{ray structure}. That is a (flat) affine structure such that the linear part is given by the products of a diagonal transformation and a commuting rotation.

We show that closed manifolds with a rank one ray structure are either complete or their developing map is a cover onto the complement of an affine subspace. This result extends the geometric picture given by Fried on closed similarity manifolds.

We prove, in the line of Markus conjecture, that if the rank one ray geometry has parallel volume, then closed manifolds are necessarily complete.
 Finally, we show that the automorphism group of a closed manifold acts non properly when the manifold is complete. 
\end{abstract}

\section{Introduction}

We are interested in the study of \emph{ray geometries}. Those are affine geometries with a constrain on the linear part. We ask, for a fixed basis  of $\R^n$, that the linear transformations $f$ verify  $f=f_mf_a=f_af_m$, with $f_m$ a rotation and $f_a$ a diagonal transformation:
\begin{equation}
f_a = \begin{pmatrix}\beta_1 \\ & \beta_2 \\ &&\ddots \\ &&& \beta_n\end{pmatrix}, \; \beta_i>0.
\end{equation}

Fried~\cite{Fried} has shown that if $f_a$ is always an homothety (the corresponding ray geometry is in fact the similarity geometry of $\R^n$) then every (flat) closed manifold is either complete or radiant: the developing map is a cover onto the complement of a point. The proof rely extensively on the fact  that homotheties are either global contractions or global expansions.
In a previous work, we extended Fried theorem to Carnot groups in general~\cite{Ale}.
But it still relies on the dynamic hypothesis that  $f_a$ either contracts globally or expands globally. It is of course very rare for affine transformations to verify this dynamic property.

Carrière~\cite{Carriere} made a study about affine manifolds that have \emph{$1$-discompacity} holonomy (any infinite sequence of  transformations can contract at most one direction). His proof is clearly linked to the ideas of Fried. The common point is the study of the dynamics associated to an incomplete geodesic.  We will call it a \emph{Fried dynamics}.

\vskip10pt
A \emph{rank one} ray geometry is a ray geometry with $f_a$ that must belong to a $1$-parameter subgroup. In general, $f_a$ does not expand or contract globally.
We show the following. 

\begin{theorem*}[\ref{thm-comp}]
\thmpartialcomp
\end{theorem*}

Associated to a classic discreteness argument from Matsumoto~\cite{Matsumoto}, it implies Fried theorem.  There remains open questions for general rank one ray geometries, notably: \emph{what are the rank one geometries that have incomplete manifolds?}
There might be ray geometries for which no closed manifold is incomplete.

\paragraph{Markus conjecture}In relation to this question we need to discuss  Markus conjecture~\cite{Markus}.
It states that \emph{closed flat affine manifolds with parallel volume are complete}.
Until now, the only  known results about Markus conjecture rely on strong conditions on the holonomy: abelian~\cite{Smillie}, nilpotent~\cite{FGH}, solvable~\cite{GH2} and distal~\cite{Fried2}.
There are also  results requiring a  Kleinian-type geometric hypothesis~\cite{Jo,Tholozan}.
An important case comes from Carrière~\cite{Carriere} under the hypothesis of $1$-discompacity.

We prove Markus conjecture for  closed manifolds having a rank one ray structure. It follows from the combination of theorem~\ref{thm-comp} and the irreducibility of the linear holonomy when there is a parallel volume from~\cite{GH}.

\begin{corollary*}[\ref{thm-markus}]
\thmmarkus
\end{corollary*}

Corollary~\ref{thm-markus} provides many new examples since the discompacity is not constrained. For instance, if we assume the linear part to be given by the diagonal transformations $(\beta x,\beta y,\beta^{-2} z)$, then it has parallel volume and has $2$-discompacity  (by taking $\beta\to 0$ both the directions $x$ and $y$ are contracted). Corollary~\ref{thm-markus} ensures that closed manifold having such structures are always complete.

\paragraph{The automorphism group}Parallel volume is an interesting condition to expect completeness. Another one is the non-compacity of the automorphism group.
It is vaguely conjectured by \cite{Gromov} that if the automorphism group is large enough, it should suffice to classify the manifold.
We prove the following.

\begin{theorem*}[\ref{thm-auto}]
\thmauto
\end{theorem*}

Interestingly,
there are counter-examples  in higher rank ray geometries. We will give an example of a radiant manifold having a ray structure  of rank two with an automorphism group  that acts non properly. 

Theorem~\ref{thm-auto} suggests that completeness with large automorphism group is a   phenomenon very special to ray geometries of rank one, in contrast to Markus conjecture where the parallel volume is the only condition.

\paragraph{Other geometries} Miner~\cite{Miner} generalized Fried theorem to the Heisenberg space instead of $\R^n$ (\cite{Ale} generalized it to every Carnot group). We expect the study of ray structures to not be too much dependent on $\R^n$ but rather on the choice of a  nilpotent space, such as the Heisenberg space. In a forthcoming preprint we examine this generalization to nilpotent spaces.

\paragraph{Organization of the paper}
In section~\ref{sec-2}, we introduce ray structures and convexity arguments.
In section~\ref{sec-3}, we study Fried dynamics: those are the dynamics associated to an incomplete geodesic. We  show theorem~\ref{thm-comp} on the completeness and incompleteness of closed manifolds. In section~\ref{sec-4} we show the completeness in the case of parallel volume (corollary~\ref{thm-markus}) and in the case of an automorphism group acting non properly (theorem~\ref{thm-auto}).

\paragraph{Acknowledgment}This work is part of the author's doctoral thesis, under the supervision of Elisha Falbel. The author is indebted and thankful to  E.~Falbel for the many hours devoted to discuss this paper.

\section{Ray structures}\label{sec-2}

\begin{definition}
Let $D\colon \widetilde M \to \R^n$ be a local diffeomorphism. A curve $\gamma\colon \iI \to \widetilde M$ is \emph{geodesic} if $D(\gamma)$ is a geodesic segment in $\R^n$.

For any $p\in \widetilde M$, define $V_p\subset \rT_p\widetilde M$  the set of the vectors such that there exists a geodesic segment $\gamma\colon \iI \to \widetilde M$ with $\gamma(0)=p$, $\gamma'(0)\in V_p$.  We say that $V_p$ is the \emph{visibility set} from (or of) $p\in \widetilde M$.
\end{definition}

\begin{definition}
Let $D\colon\widetilde M \to \R^n$ be a local diffeomorphism. A subset $C\subset \widetilde M$ is \emph{convex} if $D$ is injective on $C$ and $D(C)$ is convex. Let $x\in \widetilde M$, a point $y\in \R^n$ is \emph{visible from $x$} if there exists a geodesic segment from $x$ to $z$ such that $D(z)=y$.
\end{definition}

Following Carrière~\cite{Carriere} (see also~\cite{Benzecri,Koszul}):

\begin{proposition}[\cite{Carriere}]
Let $D\colon\widetilde M \to \R^n$ be a local diffeomorphism  for any $p\in \widetilde M$.
\begin{itemize}
\item The visible set $V_p\subset \rT_p\widetilde M$ is open.
\item $D$ is  a diffeomorphism if and only if for any $p\in\widetilde M$, $V_p= \rT_p\widetilde M$.
\item $D$ is injective on $\exp_p(V_p)$ for any $p\in \widetilde M$.
\item 
Let $C$ be convex and containing $p\in \widetilde M$. Then $C\subset \exp_p(V_p)$.
\item If for $p\in \widetilde M$, $\exp_p(V_p)$ is convex then $\exp_p(V_p)=\widetilde M$.
\item
If $C_1,C_2$ are convex subsets in $\widetilde M$ and $C_1\cap C_2\neq \emptyset$, then $D$ is injective on $C_1\cup C_2$.
\end{itemize}
\end{proposition}

\begin{definition}
A manifold $M$ is a $(G,X)$-manifold (or, equivalently, has a $(G,X)$-structure) for $G$ a group acting  (by analytic diffeomorphisms) on a space $X$ if there exists $(D,\rho)$ a pair of a local diffeomorphism $D\colon\widetilde M \to X$ (called the developing map) and a morphism $\rho\colon\pi_1(M)\to G$ (called the holonomy morphism) such that
\begin{equation}
\forall \gamma\in\pi_1(M),\forall x\in\widetilde M, \; D(\gamma\cdot x)=\rho(\gamma)D(x).
\end{equation}

 Let $\Aff(\R^n) = \R^n\rtimes \GL(\R^n)$. We say that $(\Aff(\R^n),\R^n)$ is the \emph{affine geometry} of $\R^n$ and if $M$ has a $(\Aff(\R^n),\R^n)$-structure then it is  an  \emph{affine manifold}.
\end{definition}

\begin{definition}
Let $M$ be an affine manifold. It is \emph{complete} if $D\colon\widetilde M \to \R^n$ is a diffeomorphism.
\end{definition}

\begin{definition}
A \emph{ray geometry} on the space $\R^n$ is a model geometry $(G,\R^n)$ with $G= \R^n\rtimes KA$ and $KA\subset \GL(\R^n)$  such that the following conditions are verified.
\begin{enumerate}
\item The subgroup $A$ is isomorphic to a multiplicative group $(\R_+^*)^r$, with $r$ called the \emph{rank} of the ray geometry. The subgroup $K$  is compact and centralizes $A$.
\item 
There exists a fixed basis $(e_1,\dots,e_n)$ of $\R^n$ such that
for an isomorphism $(\beta_1,\dots,\beta_r)\colon A\to (\R_+^*)^r$ there exist $d_{i,j}\in \R$ such that $A$ consists of the transformations:
\begin{equation}
\begin{pmatrix}
\beta_1^{d_{1,1}}\cdots\beta_r^{d_{r,1}} \\ 
&\ddots \\
 &&\beta_1^{d_{1,n}}\cdots\beta_r^{d_{r,n}}
\end{pmatrix}
.
\end{equation}
\end{enumerate}
\end{definition}

\paragraph{Limit sets}
Later one we will need to study sequences of subsets that have a limit.

\begin{definition}
Let $B_n$ be a sequence of  subsets. Its \emph{limit set}, denoted $\lim B_n$, is the set of the points $\lim x_n$ for sequences $x_n\in B_n$ (with $x_n$ chosen for each $B_n$).
\end{definition}

\begin{proposition}\label{prop-visconv}
Let $g_{i}\in \pi_1(M)$ be a sequence of transformations and $S\subset \widetilde M$ be closed and convex. Assume that $g_{i}S$ has a limit point $y\in \widetilde M$. Consider $B_\infty = \lim D(g_{i}S) = \lim \rho(g_i) D(S)$ the limit set in the developing map.  Then $B_\infty$ is closed, convex and there exists a closed and  convex subset $S_\infty$  containing $y$ such that $D(S_\infty)=B_\infty$.
\end{proposition}

\begin{proof}
Observe that by convexity, each $D(g_i S)$ is closed and convex. Therefore $B_\infty$ is again closed.
It is also convex. If two points $x_1,x_2$ belong to the limit set $B_\infty$ then we consider their associated converging sequences. We can form geodesics in each $g_i S$ that joins the points from each sequence. The limit of those geodesics exists, is geodesic and connects $x_1$ and $x_2$.

Similarly, $S_\infty=\lim g_i S$ is closed and convex.
We show its developing image covers $B_\infty$. Let $y_i\in g_i S$ be a sequence tending to $y$. Let $z\in B_\infty$ be the limit of $D(z_i)$ for $z_i\in g_i S$. The geodesics $\gamma_i$ from $y_i$ to $z_i$ are developed into $D(\gamma_i)$ and those tend to the geodesic from $D(y)$ to $D(z)$. This geodesic is compact (defined on $\iI$) and each $\gamma_i$ is completely visible. Therefore $\lim \gamma_i$ corresponds to a visible geodesic issued from $y$. It necessarily ends at $z$ by injectivity.
\end{proof}

\section{Fried dynamics}\label{sec-3}

Let $M$ be an incomplete closed affine manifold. Choose any metric on $M$ compatible with its topology.
Denote
 $\pi\colon \widetilde M \to M$ its universal cover.
If $x\in\widetilde M$ and $\gamma\subset \widetilde M$ is a geodesic issued from $x$, incomplete at $t=1$, then the projection of $\gamma$ in $M$ is an open curve without any continuous completion at $t=1$.
Since $M$ is closed, the projection $\pi(\gamma)$ has a recurrent point $y\in M$.

Let $U\subset M$ be a compact neighborhood of $y$, convex and trivializing the universal cover.
For the choice of a decreasing sequence $\epsilon_i \to 0$ we can define $t_i$ the time such that $\pi(\gamma(t_i))$ belongs to $U$ and is at distance at most $\epsilon_i$  to $y\in M$. We ask that $\pi(\gamma)$ exits $U$ between the times $t_i$ and $t_{i+1}$. We have that $t_i\to 1$ since the geodesic is incomplete at $t=1$.

We use the trivialization of the universal cover by $U$. It provides $U_i\subset \widetilde M$ such that $\gamma(t_i)\in U_i$. Let $y_i\in U_i\subset \widetilde M$ be the lifts of $y\in U\subset M$.
Every $U_i$ is again a compact convex neighborhood of $y_i$.
Through the developing map $D\colon \widetilde M \to \R^n$, each $D(U_i)$ is compact, convex and they accumulate along the compactification $\overline{D(\gamma)}$ of $D(\gamma)$ at  $t=1$.
Since the intersections of $\gamma$ with the $U_i$'s are transverse and disjoint along $\gamma$, the $D(U_i)$'s intersect transversally and disjointly $D(\gamma)$.

\begin{definition}
Let $\gamma\subset \widetilde M$ be an incomplete geodesic at $t=1$. Let $y\in M$ be a recurrent point of its projection into $M$. Let $U$ be a convex, compact, trivializing neighborhood of $y$. Let $\epsilon_i\to 0$ be decreasing and let
$t_i\to 1$ be an associated sequence of times such that  (for the lifts $y_i\in U_i$ of $y\in U$) we have $\gamma(t_i)\in U_i$ and the distance between $\pi(\gamma(t_i))$ and $y$ is lesser than $\epsilon_i$.
Define $g_{ji}\in \pi_1(M,y)$ the transformations of $\widetilde M$ verifiying
\begin{equation}
g_{ji}(U_i) = U_j.
\end{equation}
Those data define a \emph{Fried dynamics}.
\end{definition}

Note that a subsequence of the  times $\{t_i\}$ (or the distances $\{\epsilon_i\}$) corresponds univocally to a subsequence of the pairs  $\{y_i\in U_i\}$.

The transformations $g_{ji}$ verify a cocycle property: 
\begin{equation}
g_{ki} = g_{kj}g_{ji}.
\end{equation}
For the affine geometry $(G,\R^n)$ considered, denote by $T_{ji}$ the corresponding transformations by the holonomy morphism $T_{ji}=\rho(g_{ji})\in G$.

\begin{figure}[ht]
\centering
\includegraphics[width=0.75\textwidth]{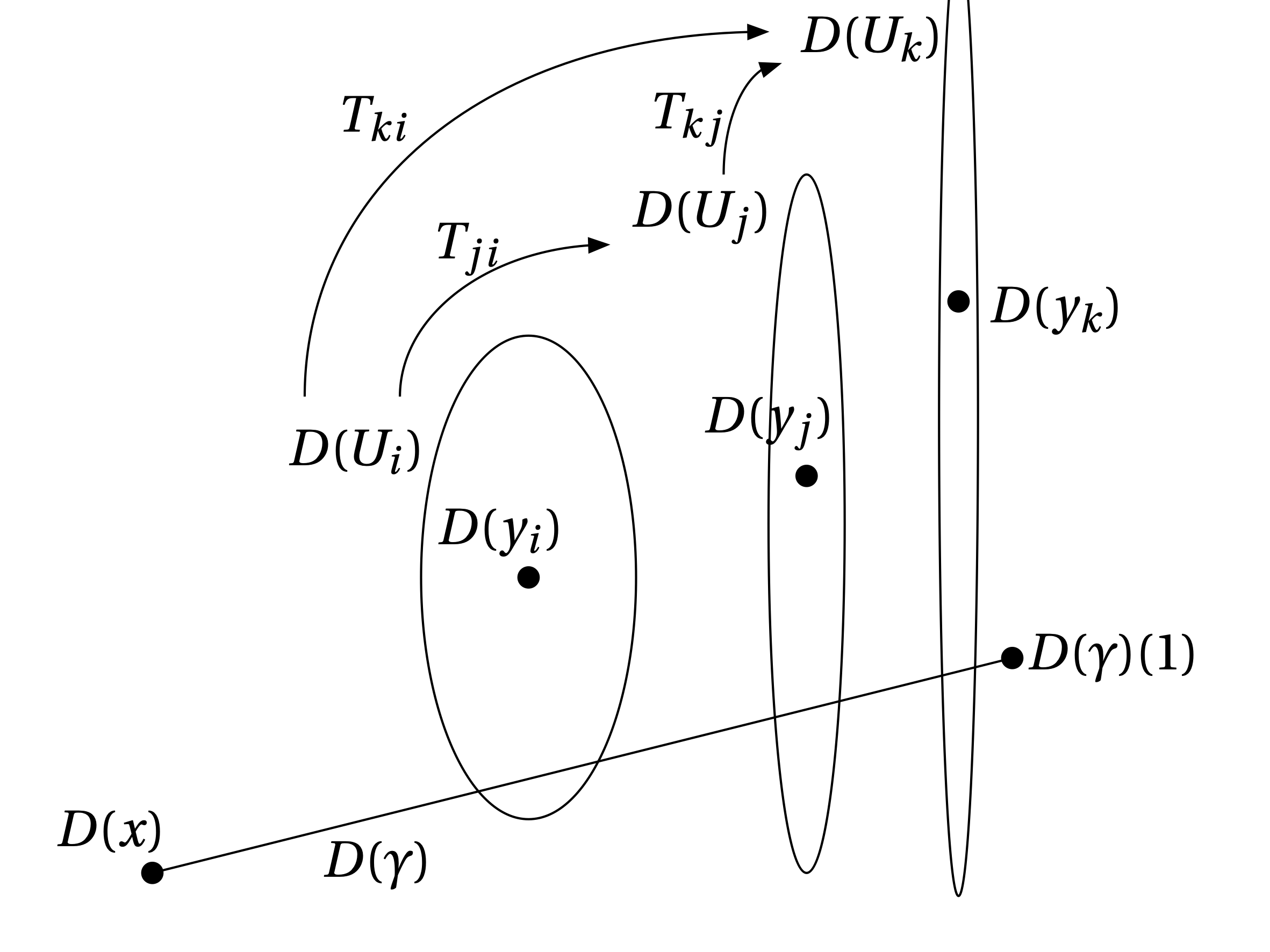}
\caption{A Fried dynamics.}
\end{figure}

\paragraph{Idea of the proof of theorem~\ref{thm-comp}}To prove theorem~\ref{thm-comp} we use a strategy comparable to what Fried~\cite{Fried} employed.

 First, we study  Fried dynamics. It is the most important step. Later we will introduce a convex subset $S$ (with smooth boundary) that contains the invisible geodesic $\gamma$. The goal of this first step will be to be able to describe what is the shape of $T_{ji}^{-1}(S)$ when $j\gg i \to \infty$. This study is independent of the rank of the ray geometry.

Then, by using the convexity argument in proposition~\ref{prop-visconv}, we show how to construct from $S$ a half-space of $\R^n$ that is completely visible. This step will depend on the fact that the ray geometry has rank $1$.

Finally, we will show how the different half-spaces obtained by varying $x$ and $\gamma$ can be combined in $D(\widetilde M)$. This will imply that the developing map avoids every invisible point, and therefore is a covering onto its image.

\vskip10pt
Now we return to the study of the Fried dynamics.
We  assume that $(G,\R^n)$ is a ray geometry.
We have a basis $(e_1,\dots,e_n)$ of $\R^n$ such that $A\subset KA\subset G$ acts diagonally. 
Once a base point of $\R^n$ is chosen, we can express any $T_{ji}\in G$ by 
\begin{equation}
T_{ji}(x) = c_{ji} + f_{ji}(x),
\end{equation}
with $c_{ji}\in \R^n$ and $f_{ji}\in KA$.
Recall that a change of the base point  is expressed by
\begin{equation}
 T_{ji}(x)=T_{ji}(y-y+x)  = (c_{ji} + f_{ji}(y)) + f_{ji}(-y+x).
\end{equation}
It preserves the linear part $f_{ji}$.
Decompose each $f_{ji}$ into 
\begin{equation}
f_{ji} = f_{ji,K}f_{ji,A}
\end{equation}
with $f_{ji,K}\in K$ et $f_{ji,A}\in A$. 
Recall that $K$ centralizes $A$, so both factors commute. The cocycle property on $T_{ji}$ implies that each factor also verifies the cocycle relation:
\begin{equation}
f_{ki,K} = f_{kj,K}f_{ji,K}, \; f_{ki,A} = f_{kj,A}f_{ji,A}
\end{equation}

\begin{lemma}
Up to  a subsequence of $\{y_i\in U_i\}$, we have
\begin{equation}
\lim_{j\to \infty}\lim_{i\to \infty} f_{ji,K} = \id.
\end{equation}
Let $e_q$ be a basis vector. Denote by $\beta_{ji,q}$ the diagonal element of $f_{ji,A}\in A$ for the direction $e_q$. Up to  a subsequence of $\{y_i\in U_i\}$, we have
\begin{equation}
\lim_{j\to \infty}\lim_{i\to \infty}  \beta_{ji,q} = \omega_q\in \{0,1,\infty\}.
\end{equation}
\end{lemma}

\begin{proof}
Since $K\subset G$ is compact, we can suppose up to a subsequence of $\{y_i\in U_i\}$ that $f_{ji,K}\to L\in K$ when $j\gg i \to \infty$. But then the cocycle property implies $L^2=L$ that can only be verified for $L=\id$. By an analogous argument, with the compactification of $\R_+$ into $\R_+\cup\{\infty\}$ we have a limit $\beta_{ji,q}\to \omega_q$ that verifies $\omega_q^2=\omega_q$, and it can only be true if $\omega_q\in \{0,1,\infty\}$.
\end{proof}

It should be noted that if
$\beta_{kj,q}\to \omega_q$ when $k\gg j\to \infty$, then we can obtain an information on how $\beta_{ki,q}$ can evolve when $i$ is fixed and $k\to \infty$. Indeed, we have
\begin{equation}
\lim _{k\to \infty}\beta_{ki,q} = \lim_{k\to\infty }\lim_{j\to\infty} \beta_{ki,q} = \lim_{k\to \infty}\lim_{j\to\infty} \beta_{kj,q}\beta_{ji,q}= \omega_q \lim_{j\to \infty}\beta_{ji,q}.
\end{equation}
If $\beta_{ki,q}\to r$ when $k \to \infty$ (up to  a subsequence) then $r=\omega_qr$. If $\omega_q\in\{0,\infty\}$ it implies $r=\omega_q$.

\begin{lemma}\label{lem-valprop}
Let $i>0$ fixed.  If $\beta_{ki,q}\to r$  and $\omega_q=1$ then $r$ must be a (finite) real positive number.
\end{lemma}

\begin{proof}
Assume that $r=0$, we show that $\omega_q=0$. For any $j>i$ large enough, there exists  $k>j$ such that $\beta_{ki,q}<\beta_{ji,q}^2<1$.  It implies that $\beta_{kj,q} = \beta_{ki,q} \beta_{ji,q}^{-1} < \beta_{ji,q}<1$. Therefore $\beta_{kj,q}$ can only tend to $0=\omega_q$. An analogous argument shows that if $r=\infty$ then $\omega_q=\infty$. 
\end{proof}

\begin{proposition}
There exists a subsequence of $\{y_i\in U_i\}$ such that:
\begin{itemize}
\item for $i>0$ fixed, $f_{ji,K}$ converges;
\item for  $i>0$ fixed, the sequence $\{\beta_{ji,q}\}$ is monotonic or constant for $j>i$.
\end{itemize}
\end{proposition}

Once this proposition is proven, we will assume that we chose such a subsequence of $\{y_i\in U_i\}$.

\begin{proof}
For the first property, by compacity of $K$, we can assume that $f_{ji,K}$ converges since it must have an accumulation point in $K$.

For the last property, it can be done for a single $i_0>0$ up to  a subsequence of $j>i_0$. 
Now let $i>i_0$. Then $\beta_{ji,q} = \beta_{ji_0,q}\beta_{ii_0,q}^{-1}$ and  it must be again monotonic or constant.
\end{proof}

Note that the last property implies that $\beta_{ji,q}\to r$  and therefore the preceding discussion applies.

\begin{definition}
We define a linear decomposition $E\oplus P \oplus F$ of $\R^n$ by deciding that $e_q\in E$ if $\omega_q=0$, $e_q\in P$ if $\omega_q=1$ and $e_q\in F$ if $\omega_q=\infty$.
\end{definition}

Since $A$ acts diagonally and $K$ centralizes $A$, we have that $f_{ji}$ preserve the decomposition $E\oplus P \oplus F$. 

Note that $\dim E>0$ since $D(U_i)$ accumulates disjointly on $\overline{D(\gamma)}$. The other subspaces might be reduced to $\{0\}$.

If we chose any base point $p$, it gives three affine subspaces based at $p$:
\begin{equation}
\forall L\in \{E,F,P\} \forall p\in \R^n, \; L|_p \coloneqq p + \exp(L).
\end{equation}

\vskip10pt
We come back to $M$ and discuss the choice of $U\subset M$.
Choose $V\subset U$ a smaller neighborhood of $y$. Since we chose distances $\epsilon_i\to 0$, for $i\geq i_0$ large enough, every $\pi(\gamma(t_i))$ will also belong to $V$. Note that we can lift $V\subset M$ into $V_i\subset U_i\subset \widetilde M$.

\begin{lemma}\label{lem-geoy}
For any $i>0$, $g_{ji}^{-1}\gamma$ has $y_i$ for limit point when $j\to \infty$.
\end{lemma}
\begin{proof}
Let $i>0$. Choose any neighborhood $V_i$ of $y_i$ contained in $U_i$. Then $V_i$ corresponds to a neighborhood $V\subset U$ of $y$ in $M$. Therefore for $j$ large enough, $V_j$ intersects $\gamma$ and therefore $V_i =g_{ji}^{-1}V_j$ intersects $g_{ji}^{-1}\gamma$.
\end{proof}

\begin{definition}
In $M$, let $U_1 = U$. Define for a sequence of $\{0<r<1\}$ tending to $0$ a sequence $\{U_r\}$ of compact convex neighborhoods  of $y\in M$. Assume that $U_r$ is decreasing for the inclusion and that $U_r\to \{y\}$ when $r\to 0$. In $\widetilde M$, define $U_{1,i}=U_i$ and  $U_{r,i}$ the lift of $U_r$ such that $U_{r,i}\subset U_{1,i}$. In $D(\widetilde M)$, define
\begin{equation}
C_{r,i}  = D(U_{r,i}).
\end{equation}
\end{definition}

Recall that $U_1=U$ is a convex, compact  and trivializing neighborhood of $y\in M$.
Every $C_{r,i}$ is convex and compact in $\R^n$ since $U_{r}$ is a convex, compact and trivializing neighborhood of $y$ in $M$. By construction:

\begin{lemma}
For any $i,j$ and $r\geq 0$, $g_{ji}U_{r,i} = U_{r,j}$ and by consequence $T_{ji}C_{r,i}=C_{r,j}$.\qed
\end{lemma}

As discussed before, since $U_r\subset U$ is a neighborhood of $y$, $D(\gamma)(1)$ is a limit point of $\{C_{r,i}\}$ since for $j>j_0$  large enough $D(\gamma(t_j))$ belongs to $C_{r,j}$.

\begin{proposition}
Let $r>0$. Any accumulation point of $\{C_{r,i}\}$ is a limit point.
\end{proposition}
\begin{proof}
Choose any base point in $\R^n$.
We know that $D(\gamma)(1)$ is a limit point. Let $y_j\in C_{r,j}$ be tending to $D(\gamma)(1)$. Each $y_j\in C_{r,j}$ can be written $c_{ji}+f_{ji}(x_{i,j})$ for $x_{i,j}\in C_{r,i}$. We can write any $z\in C_{r,i}$ as $x_{i,j} - x_{i,j}+z$ so that with $L_{r,i} = - x_{i,j} + C_{r,i}$ we have:
\begin{align}
C_{r,j}=T_{ji}(C_{r,i}) &= c_{ji} + f_{ji}(C_{r,i}) \\
& = (c_{ji} + f_{ji}(x_{i,j})) + f_{ji}(-x_{i,j}+ C_{r,i}) \\
&= y_j +  f_{ji}(L_{r,i}) .
\end{align}

Therefore, since $y_j$ converges,
an accumulation point $z_{\sigma(j)}\to z$ of the sequence $\{C_{r,j}\}$ corresponds to an accumulation point $w_{\sigma(j)} \to w$ of $\{f_{ji}(L_{r,i})\}$.
So we only need to prove that $\{f_{ji}(L_{r,i})\}$ has for limit set its set of accumulation points. Note that $0\in L_{r,i}$ and $L_{r,i}$ is convex and compact. (The point $0\in L_{r,i}$ corresponds in fact to $y_j$ tending to $D(\gamma)(1)$.)

Let $i>0$ and consider $j\geq j_0$. Then $f_{ji}$ has its rotational part that converges and its diagonal factors that converges by monotonic or constant values. Recall that $K$ centralizes $A$. So the action by the rotational part, $f_{ji,K}(L_{r,i})$, has a limit, say $L_K$.
Now, the diagonal action is monotonic (or constant) in each direction, therefore any accumulation point of $f_{ji,A}(L_K)$ is in fact a limit point.
\end{proof}

\begin{definition}
For any fixed value $r> 0$, let $C_{r,\infty}$ be the limit set  of the sequence $\{C_{r,i}\}$.
\end{definition}

Note that when $r=0$, each $C_{r,i}$ is reduced to $D(y_i)$. Those have no reason to accumulate to $D(\gamma)(1)$.  This is why we asked $r$ to be different from $0$.

\begin{definition}
Define 
\begin{equation}
C_{0,\infty} = \bigcap_{r>0} C_{r,\infty}.
\end{equation}
\end{definition}

This definition makes sense because the $C_{r,\infty}$ are decreasing for the inclusion when $r\to 0$.

\begin{lemma}
All the sets $C_{r,\infty}$ and $C_{0,\infty}$ are convex.
The set $C_{0,\infty}$ is nonempty (it contains $D(\gamma)(1)$) and is an affine subspace.
\end{lemma}
\begin{proof}
Because every $C_{r,j}$ is convex, so are the limits $C_{r,\infty}$ and $C_{0,\infty}$.
The fact that $D(\gamma)(1)\in C_{0,\infty}$ is clear. Now, observe that if $C_{0,\infty} = \{D(\gamma)(1)\}$ then the statement is true.

The set $C_{0,\infty}$ is closed and convex. To show that it is an affine subspace, we show that any maximal non-zero geodesic in $C_{0,\infty}$ is always open. Let $\gamma\subset C_{0,\infty}$. For any $r>0$ there exists $z_i\in C_{r,i}$ and $w_i\in C_{r,i}$ such that the geodesic $\eta_i$ from $z_i$ to $w_i$ tends to  a geodesic containing (or equal) to $\gamma$. But then $\eta_i$ is always contained in $C_{r+\epsilon,i}$ and can be extended to a strictly larger geodesic $\mu_i$. Then $\mu_i$ tends to a geodesic in  $C_{r+\epsilon,i}$ strictly containing $\gamma$. Because this can be done for any $\epsilon>0$ small enough, $\gamma$ is necessarily open if it is maximal.
\end{proof}

\begin{lemma}
Choose $D(\gamma)(1)$ as base point of $\R^n$.
\begin{align}
\forall r>0, \; C_{r,\infty}&\subset \left(P\oplus F\right)|_{D(\gamma)(1)}\\
C_{0,\infty}&= F|_{D(\gamma)(1)}.
\end{align}
\end{lemma}

\begin{proof}
Choose  a $C_{r,i}$ for $i$ large enough, so that $T_{ji}$ tends to a transformation with a linear part with almost no rotation and diagonal elements tending to $\omega\in\{0,1,\infty\}$. 
The limit $T_{ji}C_{r,i}=C_{r,j}\to C_{r,\infty}$ can only contain geodesics in the directions where $\omega\neq 0$. So it is in $F\oplus P$ based at $D(\gamma)(1)$. If $r\to 0$, then the directions of $P$ are contracted in $C_{r,\infty}$, so  only  $F$ remains.
\end{proof}

\paragraph{Fixed point}Now we return to the general study. The next step is to find an asymptotic fixed point for $T_{ji}$.

We have that for any $x\in \R^n$, there exists a unique decomposition $x = x_L + x_F$ with $x_L\in E\oplus P$ and $x_F\in F$.

\begin{lemma}
Choose $D(\gamma)(1)$ as base point and decompose $T_{ji}(x)=c_{ji} + f_{ji}(x)$ with $f_{ji}\in KA$ and $c_{ji}\in \R^n$. 
Let $c_{ji} = c_{ji,L}+c_{ji,F}$ be the decomposition following $\R^n= L\oplus F$. Define
 $Q_{ji}(x)= c_{ji,F} + f_{ji}(x)$.  Then
\begin{equation}
\lim_{j\to \infty} T_{ji}(x)-Q_{ji}(x) = \lim_{j\to\infty} c_{ji,L} =0
\end{equation}
 and therefore this convergence is uniform for $x\in \R^n$. 
\end{lemma}

\begin{proof}
For any $\epsilon >0$, 
$T_{ji}(C_{\epsilon,i})$ has for limit $C_{\epsilon,\infty}$ when $j\to \infty$. 
Consider the linear decomposition $L\oplus F$ based at $D(\gamma)(1)$.
When $\epsilon\to 0$ and $j\to \infty$,  the coordinates of
 $T_{ji}(C_{\epsilon,i})$ all tend to $0$ in the linear subspace $L\subset \R^n$. Indeed, $C_{0,\infty}=F|_{D(\gamma)(1)}$.
 
Note that
 $c_{ji,L}$ tends or not to $0$ when $j\to \infty$ independently from  the choice of $\epsilon$. 

Now decompose $T_{ji}(C_{\epsilon,i})$, and note that $f_{ji}$ preserves the decomposition $L\oplus F$.
\begin{align}
c_{ji} + f_{ji}(C_{\epsilon,i}) &= c_{ji,L}+c_{ji,F} + f_{ji}(C_{\epsilon,i})_L + f_{ji}(C_{\epsilon})_F \\
&= \left(c_{ji,L} + f_{ji}(C_{\epsilon,i})_L\right) + \left(c_{ji,F} + f_{ji}(C_{\epsilon,i})_F\right)
\end{align}
The term  $c_{ji,F} + f_{ji}(C_{\epsilon,i})_F$ is exclusively in $F$.
Hence the coordinate in $L$ is  determined by $c_{ji,L} + f_{ji}(C_{\epsilon,i})_L$  and it  must tend to $0$ when $j\to \infty$ and $\epsilon\to 0$.

But for any $j>i$, $f_{ji}$ acts as a contraction on $(C_{\epsilon,i})_L$.  Hence, for any $j>i$ fixed, when $\epsilon\to 0$ we have $f_{ji}(C_{\epsilon,i})_L\to 0$. By consequence, when $j\to \infty$ and $\epsilon\to 0$, $c_{ji,L}\to 0$. It proves the lemma.
\end{proof}

Note that $Q_{ji}$ has a fixed point on $F$, since $F$ is preserved and $f_{ji}$ acts by expansions on it. 

\begin{lemma}
Denote $q_{ji}\in F$ the fixed point of $Q_{ji}$.
 For $i>0$, $q_{ji}$ converges when $j\to \infty$ and we denote  $q_i=\lim q_{ji}$. Also, $D(y_i)\in E|_{q_i}$.
\end{lemma}
\begin{proof}
Choose $q_{ji}$ as base point of $\R^n$ when we consider the transformation $T_{ji}$. We start by showing that $D(y_i)$ is an accumulation point of $E|_{q_{ji}}$.

Suppose that $D(y_i)$ is not an accumulation point of the stable subspace 
 $(E\oplus P)|_{q_{ji}}$.
 Then there exists $\epsilon>0$ so that
 $C_{\epsilon,i}$
 avoids every 
$(E\oplus P)|_{q_{ji}}$ for $j$ large enough. 
But then
 $T_{ji}C_{\epsilon,i}=C_{\epsilon,j}$ 
 escapes to infinity and cannot accumulate at $D(\gamma)(1)$, impossible.

If $D(y_i)$ is an accumulation point of
$(E\oplus P)|_{q_{ji}}$ then $D(y_i)$ is an accumulation point of $E|_{q_{ji}}$ since $C_{\epsilon,j}$ accumulates at $F|_{q_{ji}}=C_{0,\infty}$ and therefore  has arbitrarily small coordinates along $P$.

Now we show that $q_{ji}$ does not tend to infinity. 
Since $E|_{q_{ji}}$ accumulates at $D(y_i)$, it comes that $q_{ji}$ must asymptotically be at the intersection of $E|_{D(y_i)}=\lim E|_{q_{ji}} $ with $C_{0,\infty}=F|_{q_{ji}}$. But this intersection is reduced to a single point, so $q_{ji}\to q_i$.
\end{proof}

\begin{figure}[ht]
\centering
\includegraphics[width=0.75\textwidth]{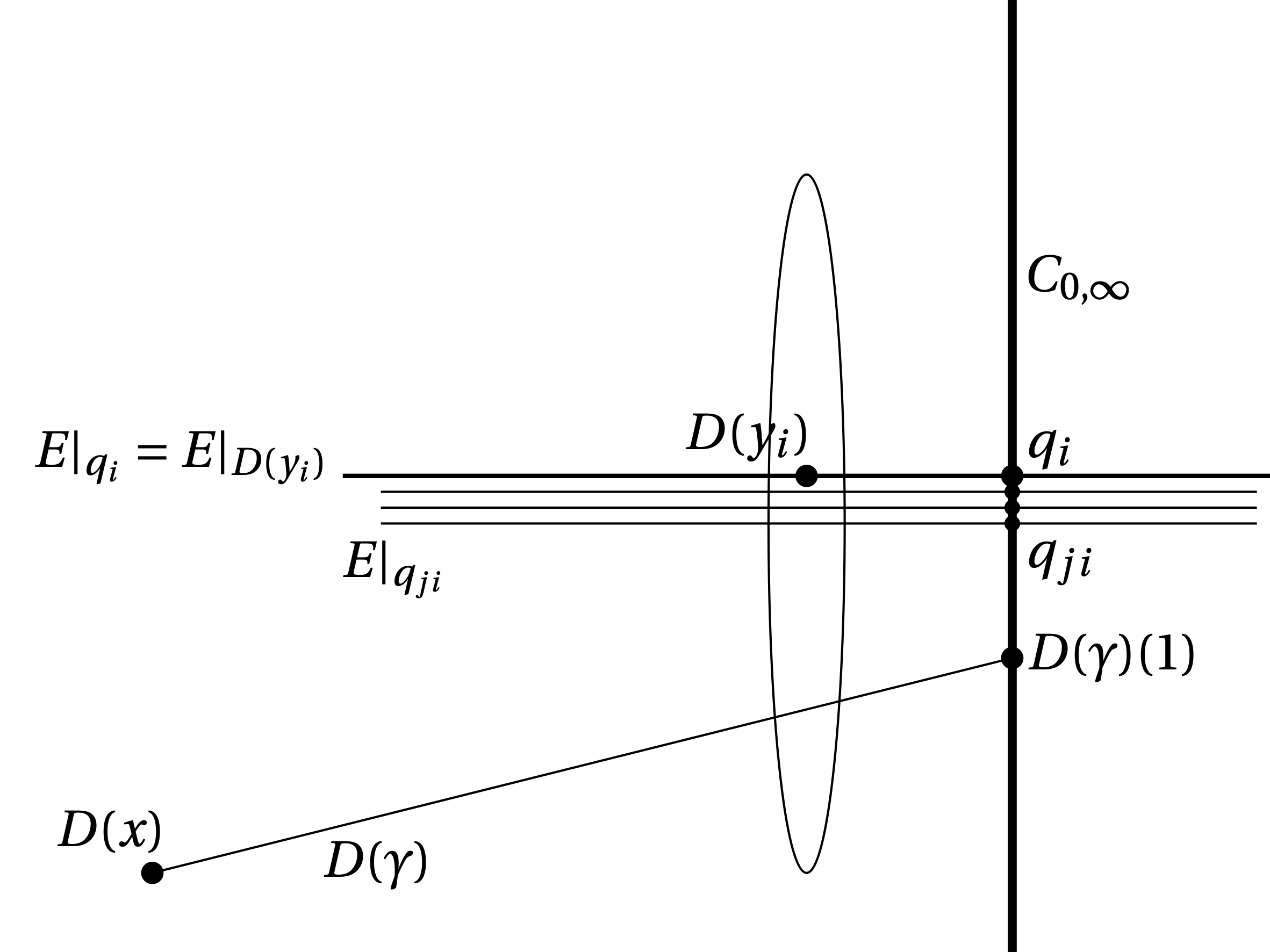}
\caption{The  dynamics of $Q_{ji}$.}
\end{figure}

\paragraph{Asymptotic dynamics}
Now we can determine the limits of the orbits in positive time, and apply proposition~\ref{prop-visconv}.

\begin{lemma}
Let $z\in E|_{q_i}$ and $V$ a neighborhood  of $z$. Then $\{T_{ji}(V)\}$ has $F|_{q_i}$ in its limit set.
\end{lemma}
\begin{proof}
It is the case for $Q_{ji}(V)$ since $V$ intersects every $E|_{q_{ji}}$ for $j$ large enough,  and we know that $T_{ji}-Q_{ji}\to 0$.
\end{proof}

\begin{proposition}\label{prop-eplus}
Let $S\subset \widetilde M$ be a convex containing the incomplete geodesic $\gamma$. Assume that $D(S)$ has a smooth boundary at $D(\gamma)(1)$. Let $i>0$. We have the following properties. 
\begin{itemize}
\item
The orbit $T_{ji}^{-1}(D(S))$ tends  to a product $E_{+,i}\times P_c$ of a half-space $E_{+,i}\subset E|_{q_i}$ and a neighborhood $P_c\subset P|_{q_i}$ of the origin.
\item
The product $E_{+,i}\times P_c$ is visible from $y_i$.
\item
The boundary of $E_{+,i}$ is  described by the limit of
 $T_{ji}^{-1}(\rT_{D(\gamma)(1)}D(S)) \cap E|_{q_i}$.
\end{itemize}
\end{proposition}

\begin{proof}
To prove this statement,
we apply the preceding lemma with 
 proposition~\ref{prop-visconv} which allows to prove visibility of the limit of convex subsets.
 
For $i>0$ fixed, we consider the convex subsets $T_{ji}^{-1}(D(S))$. We already know that $D(y_i)$ is a limit point since $\gamma\subset S$ (see lemma~\ref{lem-geoy}).  Therefore the limit of $T_{ji}^{-1}D(S)$ is a closed convex subset visible from $y_i$.

Consider $z\in D(S)$ and write $z = z_E + z_P + z_F$ with $z_P \in P$ and $z_F \in F$. Then
\begin{equation}
T_{ji}^{-1}(z) = T_{ji}^{-1}(z_E) + f_{ji}^{-1}(z_P) + f_{ji}^{-1}(z_F).
\end{equation}
Note that, up to a rotation, $f_{ji}^{-1}(z_P)\to r z_P$  (with $r>0$ as in lemma~\ref{lem-valprop}) and $f_{ji}^{-1}(z_F) \to 0$.

This observation shows that
the limit of $T_{ji}^{-1}D(S)$ has vanishing coordinates on $F$  and almost identical coordinates on $P$ (up to a rotation), say $P_c$.

Now in the coordinate of $E$,
 a point $z\in E|_{q_i}$ is a limit point of 
$T_{ji}^{-1}(D(S))$ if any neighborhood  $V$ of $z$ has $T_{ji}(V)$ accumulating at $D(\gamma)(1)$ from inside $D(S)$. This property   is described by the  relative position of $T_{ji}(V)$ to $\rT_{D(\gamma)(1)}D(S)$. It describes a half-space $E_{+,i}$ in $E|_{q_i}$.

Therefore the limit is $E_{+,i}\times P_c$ and the listed properties follow.
\end{proof}

To get a better description of $E_{+,i}$ inside $E|_{q_i}$, we must explain how to approximate $T_{ji}^{-1}$ relatively to $Q_{ji}^{-1}$. Note that \emph{a priori}
\begin{equation}
T_{ji}^{-1}(x) -Q_{ji}^{-1} (x)= - f_{ji}^{-1}(c_{ji,L}) 
\end{equation}
and might not be tending to zero, or even stay bounded.

The study of $-f_{ji}^{-1}(c_{ji,L})$ is technical. With the two following lemmas, we   will later prove  that $-f_{ji}^{-1}(c_{ji,L})\to 0$ when $j\gg i\to \infty$.

\begin{lemma}
Assume that $E_{+,i}$ does not contain $q_i$. 
For any $i>0$, there exists $M>0$ such that for any $j>i$, 
\begin{equation}
c_{ji,L} = f_{ji}(b_{ji,E}) + b_{ji,P}
\end{equation}
 with $b_{ji,E}\in E$, $b_{ji,P}\in P$ verifying $b_{ji,P}\to 0$ and $\|b_{ji,E}\|<M$.
\end{lemma}
\begin{proof}
We consider the orbit of $D(\gamma)$. We know that $T_{ji}^{-1}D(\gamma)$ has in its limit set $D(y_i)\in E_{+,i}$. Furthermore,
\begin{equation}
T_{ji}^{-1}(D(\gamma)) = -f_{ji}^{-1}(c_{ji,L}) + Q_{ji}^{-1}(D(\gamma))
\end{equation}
and $Q_{ji}^{-1}(D(\gamma))$ tends to an half-line $\Delta$ based at $q_{i}\in E|_{q_i}$. Therefore, in order for $T_{ji}^{-1}(D(\gamma))$ to accumulate at $D(y_i)\in E|_{q_i}$, one must have that $-f_{ji}^{-1}(c_{ji,L})$ is either bounded or the sum of a $P$-component tending to zero and  a point along the opposite half-line  $-\Delta$ escaping to infinity. In the latter case, it would imply that $q_i=\Delta(0)$ belongs to $E_{+,i}$.
\end{proof}

We will prove in the case of a rank one ray geometry that $E_{+,i}$ never contains $q_i$. So in fact $b_{ji,E}$ is bounded.
The next lemma explains that in fact $T_{ji}^{-1}$ is as well approximated by $Q_{ji}^{-1}$ as $i>0$ gets large.

\begin{lemma}\label{lem-approxneg}
Assume that for any $i$, $E_{+,i}$  never contains $q_i$. Then 
\begin{equation}
\lim_{k\to \infty}\lim_{j\to \infty}b_{kj,E} = 0.
\end{equation}
\end{lemma}
\begin{proof}
The cocycle relation $T_{kj}T_{ji}=T_{ki}$ gives:
\begin{align}
T_{ji} &=  f_{ji}(b_{ji,E}) + b_{ji,P} + c_{ji,F}+f_{ji} \\
T_{kj}T_{ji} &= f_{kj}(b_{kj,E})  + f_{kj}(f_{ji}(b_{ji,E}))+ b_{kj,P} + f_{kj}(b_{ji,P}) +c_{kj,F} + f_{kj}(c_{ji,F}+f_{ji}) \\
&= f_{ki}(b_{ki,E}) + b_{ki,P} + c_{ki,F} + f_{ki}
\end{align}
and it implies by identification:
\begin{align}
f_{ki}(b_{ki,E})  &=  f_{kj}(b_{kj,E})  + f_{kj}(f_{ji}(b_{ji,E})) \\
&= f_{kj}(b_{kj,E})  +  f_{ki}(b_{ji,E}) \\
b_{ki,E} - b_{ji,E} &= f_{ji}^{-1}(b_{kj,E})
\end{align}
and since $\|b_{ki,L} - b_{ji,L}\| < 2M$, we have that $b_{kj,E}$ must tend to zero since otherwise $f_{ji}^{-1}(b_{kj,E})$ would escape to infinity.
\end{proof}

\subsection*{Rank one ray  manifolds}

We examine closed manifolds with a rank one ray geometry. For those manifolds, the holonomy takes its values in $G_1=\R^n\rtimes KA_1$, where $A_1$ has dimension $1$.

\begin{lemma}
If $(G_1,\R^n)$ is a rank one ray geometry, then there exists a decomposition $L_1\oplus L_2\oplus L_3$ such that for any Fried dynamics, $P = L_3$ and $\{E,F\}=\{L_1, L_2\}$.
\end{lemma}
\begin{proof}
Because the rank is one, each direction $e_i$ has  in $A_1$ a single diagonal factor $\beta^{d_i}$ with degree $d_i\in \R$.
Let $L_1$ be generated by the vectors $e_i$ such that $d_{i}<0$, $L_2$ when $d_i>0$ and $L_3$ when $d_i=0$. Note that if $\beta$ is exchanged with $\beta^{-1}$ then $L_1$ becomes $L_2$ and conversely ($L_3$ is unchanged).

If $E\oplus P \oplus F$ is a Fried dynamics, then we only need to show that $P=L_3$ since then it is clear that $\{E,F\}$ must be equal to $\{L_1,L_2\}$. But in rank one, if a diagonal factor $\beta^{d_i}$ tends to $1$ correspondingly to a direction $e_i\in P$, then either $d_i=0$ or $\beta\to 1$. If $d_i=0$ then $e_i\in L_3$. If $\beta\to 1$ then it is true globally on $\R^n$, showing that $E=\{0\}$, impossible.
\end{proof}

A consequence of this observation is that $P$ is independent from the dynamics, since $d_i=0$ is a condition on $A=A_1$. Therefore, from a Fried dynamics to another one $E$ and $F$ are either the same or exchanged. Also, every direction in $P$ is completely visible by the following lemma.

\begin{lemma}
The direction vector of an incomplete  geodesic $D(\gamma)$ has a non-vanishing coordinate along $E$ in the linear decomposition $\R^n=E\oplus P \oplus F$ associated to its Fried dynamics.
\end{lemma}
\begin{proof}
Note that $D(U_i)$ accumulates disjointly on $\overline{D(\gamma)}$. Indeed, we asked that $\pi(\gamma)$ exits $U$ in $M$ between the times $t_i$ and $t_{i+1}$.  So it is again the case in the developing map: $D(\gamma)$ exits $D(U_i)$ before entering into $D(U_{i+1})$. 

But if the direction vector of $D(\gamma)$ vanishes on $E$ then the infinite number of subsets $D(U_j)=T_{ji}D(U_i)$ cannot intersect the relatively compact geodesic $D(\gamma)$ in such a way. Indeed, $D(\gamma)$ is relatively compact and by intersecting along $D(\gamma)$, $T_{ji}(U_i)\cap D(\gamma)$ would not tend to a point, and therefore can not let $D(\gamma)$ exits $U_j$ before entering into $T_{ki}(U_i)\cap D(\gamma)$.
\end{proof}

\begin{proposition}
Let $S\subset \widetilde M$ be a convex containing $\gamma$. Assume that $D(S)$ has smooth boundary at $D(\gamma)(1)$.
Let $i>0$.
The subspace $H_i = E_{+,i}\times (P\oplus F)|_{y_i}$ is visible from $y_i$  (this is a half-space of $\R^n$).
\end{proposition}

\begin{proof}
By  proposition~\ref{prop-eplus}, $E_{+,i}$ is a visible  from $y_i$. Since the directions in $P$ are always complete, the product $E_{+,i}\times P|_{y_i}$ is fully visible from $y_i$. We show that we can extend $E_{+,i}\times P|_{y_i}$ to a visible open $H_i$ containing $E_{+,i}\times (P\oplus F)|_{y_i}$.

Let $K\subset (E \oplus P|_{y_i})$ be convex and visible from $y_i$. Since the visible space from $y_i$ is open, there exists an open convex $W\subset F$ such that $K\times W$ is visible and convex. We can extend $K$ to $E_{+,i}\times P|_{y_i}$. Indeed, otherwise, there exists $\eta$ an incomplete geodesic parallel to $E_{+,i}\times P|_{y_i}$ with endpoint in $K\times \{w\}$ and with $w\in W$. Because we are in rank one,  its limit $C_{0,\infty}'$ must be parallel to $F|_{y_i}$. Therefore it intersects $E_{+,i}\times P|_{y_i}$ and is simultaneously invisible since it is in $C_{0,\infty}'\cap (K\times W)$ and visible since $E_{+,i}\times P|_{y_i}$ is, a contradiction.

Now, $(E|_+\times P|_{y_i})\times W$ can be extended to $H_i$ such that it contains $E_{+,i}\times P|_{y_i}\times F|_{y_i} = E_{+,i}\times (P\oplus F)|_{y_i}$. Indeed, apply $T_{ji}$ for $j\to \infty$ (note that $T_{ji}(x+y) = T_{ji}(x)+f_{ji}(y)$):
\begin{align}
\lim_{j\to \infty}  T_{ji}((E|_+\times P|_{y_i})\times W) 
&=(E|_+\times P|_{y_i}) \times \lim_{j\to\infty}f_{ji}(W) \\
&= (E|_+\times P|_{y_i}) \times F|_{y_i}.\qedhere
\end{align}
\end{proof}

\begin{lemma}
The half-spaces $H_i$ tend to a half-space denoted $H_x$ when $i\to \infty$. For $i\to \infty$, $q_i$ gets closer to $\overline{H_i}$. 
\end{lemma}

\begin{proof}
For each $H_i$, note that $D(\gamma)(1)\in q_i + F$. Since $D(\gamma)(1)$ is invisible, $H_i$ cannot contain $q_i$ in its interior. Therefore lemma~\ref{lem-approxneg} applies. When $i\to \infty$, by the study of $c_{ji,L}$:
\begin{equation}
\lim_{j\to\infty}T_{ji}^{-1}(D(\gamma)(1)) = \lim_{j\to\infty} - f_{ji}^{-1}(c_{ji,L}) + q_{ji} - f_{ji}^{-1}(q_{ji}) = q_i
\end{equation}
is a point of $\overline{H_x}$.
\end{proof}

\begin{lemma}
Let $x\in\widetilde M$. Let $D(S)\subset D(\widetilde M)$ be  the maximal Euclidean open ball such that $x\in S$ and $S$ is visible from $x$. Then $H_x = \lim H_i $ contains $x$.
\end{lemma}

\begin{proof}
Consider the Euclidean metric that makes $(e_1,\dots,e_n)$ orthonormal.
We can always consider the family of open balls from $D(x)$ that are visible from $x$. There is a maximal one $S$ such that $\gamma\subset S$ is incomplete at $t=1$.

The set $H_x$ is a product $E_+\times (P\oplus F)$ with $E_+$ determined by $\lim T_{ji}^{-1}(S)$. Therefore the point $x$ belongs to $H_x$ if we can show that $T_{ji}(x)$ belongs to $S +  (P\oplus F)$ if $j\gg i$ are large enough.

By the preceding study, with $D(\gamma)(1)$ as base point,
\begin{align}
T_{ji} &= f_{ji}(b_{ji,E}) + b_{ji,P} + Q_{ji}, \\
Q_{ji} &= q_{ji} - f_{ji}(q_{ji}) + f_{ji}.
\end{align}

We want to show that $T_{ji}(x)$ belongs to $S + (P\oplus F)$, so in fact we prove that the $E$-coordinate of $T_{ji}(x)$ belongs to the euclidean ball (in $E$) $\|-x_E+y_E\|_E< \|x\|$.
With $y=T_{ji}(x)$ (note that $b_{ji,P}$ and $q_{ji}$ belong to $P\oplus F$) we obtain:
\begin{equation}
-x_E + T_{ji}(x)_E =  - x_E  +  f_{ji}(b_{ji,E}) + f_{ji}(x_E).
\end{equation}

We can now estimate on each subspace on which $f_{ji}$ acts by (almost) homotheties (note that the rotational part tends to the identity). On a coordinate $e_m$ of $E$ , $f_{ji}$ acts like $\lambda_{ji}^{d_m}$, hence on the coordinate $e_m$:
\begin{align}
\|f_{ji}(b_{ji,E}) + x_E - f_{ji}(x_E)\|_{e_m} &\leq \|f_{ji}(b_{ji,E})\|_{e_m} + \|x_E - f_{ji}(x_E)\|_{e_m}\\
&\leq \lambda_{ji}^{d_m} \|b_{ji,E}\|_{e_m} + (1-\lambda_{ji}^{d_m})\|x_E\|_{e_m}
\end{align}
Since $b_{ji,E}\to 0$, the inequality $\|f_{ji}(b_{ji,E}) + x_E - f_{ji}(x_E)\|_{e_m} <\|x_E\|_{e_m}$ is verified on each coordinate $e_m$, so the required (global) inequality on $E$ is verified  by summing the squares.
\end{proof}

For any $x\in \widetilde M$, we choose $S\subset \widetilde M$ the maximal convex open subset such that $D(S)$ is a Euclidean open ball. By applying the construction to the point $x$, the convex $S$ and an incomplete geodesic $\gamma\subset S$ (which exists by  maximality of $S$) we obtain an half-space $H_x=\lim H_i$ and $D(x)\in H_x$. This choice is now assumed.

\begin{lemma}
Let $I\subset \partial H_x$ be the invisible set from the interior of $H_x$. (Note that $(F\oplus P)|_{D(\gamma)(1)}\subset I$.) 
Then $I$ does not depend on $x\in \widetilde M$ and  $I$ is an affine subspace.
\end{lemma}
\begin{proof}
Let $D(q)\in \partial H_x$ be visible from $x\in\mathcal H_x$.  Consider $H_q$ a half-space corresponding to $D(q)$ and containing $D(q)$. 

Both half-spaces $D^{-1}(H_x)$ and $D^{-1}(H_q)$ are  convex and intersect. Therefore the developing map is injective on the union.
The invisible set from $H_x$ must be invisible from $H_q$ and vice-versa. But under  a large $T_{ji}$ (that leaves asymptotically stable $H_x$) it can only be possible if $I$ is common to both since $T_{ji}H_q$  would contain in its interior invisible points of $H_q$. (See figure~\ref{fig-3}.)

It shows also that $I$ is an affine subspace because it is the intersection of a finite number of half-spaces.
\end{proof}

\begin{figure}[ht]
\centering
\includegraphics[width=0.75\textwidth]{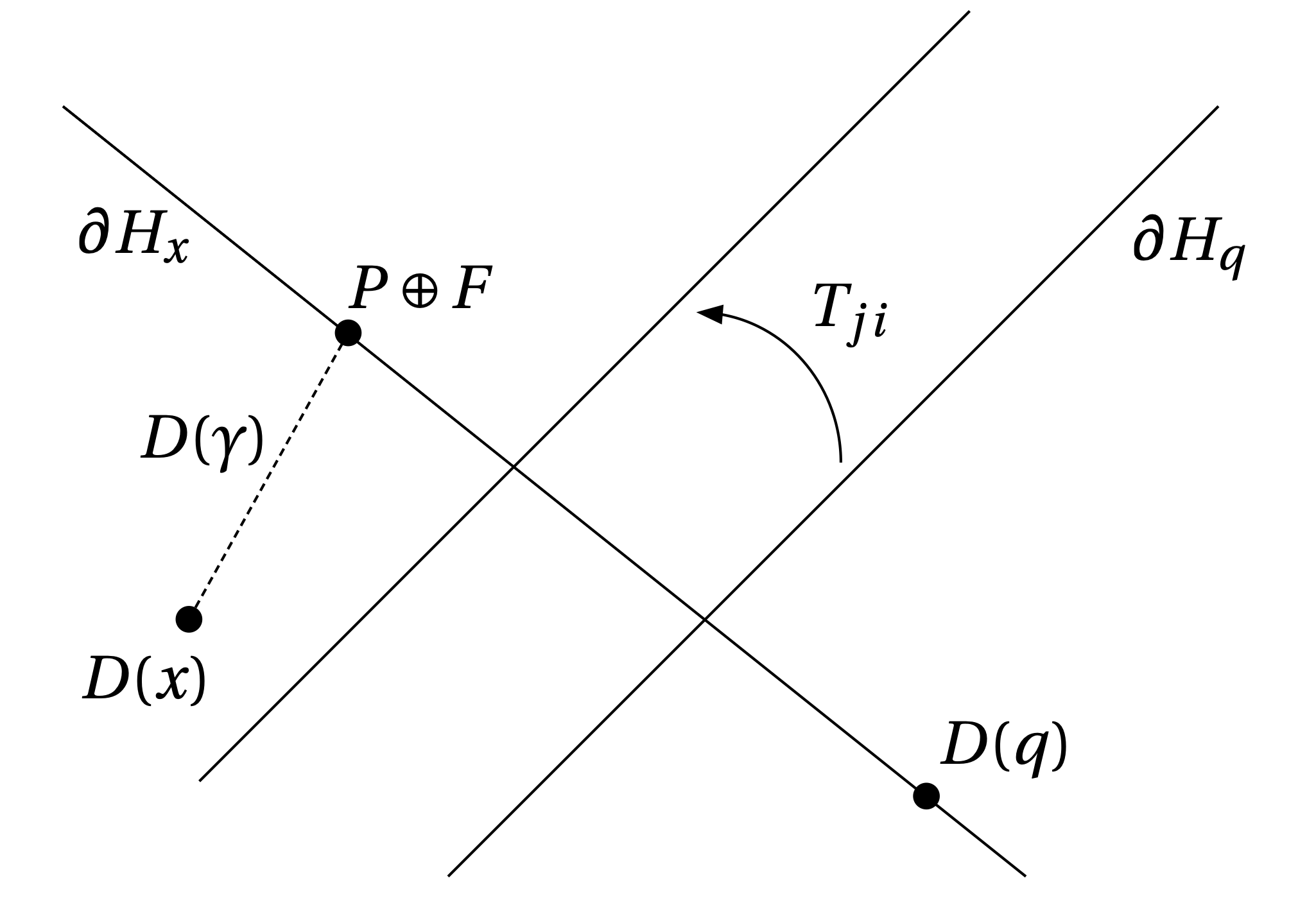}
\caption{The invisible  subspace $I$.}\label{fig-3}
\end{figure}

\begin{theorem}\label{thm-comp}
\thmpartialcomp
\end{theorem}

\begin{proof}
The affine subspace $I$ is constant and contains every $D(\gamma)(1)$ for any incomplete geodesic $\gamma\subset S$ in the maximal Euclidean open ball of any $x\in \widetilde M$. We show it implies that $D\colon\widetilde M \to \R^n-I$ is a covering map.

Let $\delta\colon \iI \to \R^n-I$ be a path. Choose $x\in \widetilde M$ such that $D(x)=\delta(0)$. We need to prove that $\delta$ can be lifted to a path in $\widetilde M$ based at $x$. We can assume that $\delta$ can be lifted for $t<1$. We show that is can be lifted at $t=1$. Since $\delta(1)\not \in I$, there exists a point $\delta(s)$ with $s<1$ such that the maximal  Euclidean open ball based at $\delta(s)$ and avoiding $I$ contains $\delta(t)$ for $s\leq t \leq 1$. But then for the lift at time $s$, the corresponding open ball $S$ is convex and allows to lift $\delta$ for $s\leq t \leq 1$.
\end{proof}

To be more precise on the nature of $I$, we use a
discreteness argument very close to what Matsumoto~\cite{Matsumoto} proposed (see also~\cite[end of sec. 4]{Ale}). 

\begin{proposition}
Let $(G_1,\R^n)$ be a rank one ray geometry such that for any Fried dynamics $F=\{0\}$. The invisible subspace avoided by the developing map is $I=P|_w$ for any Fried dynamics associated to an invisible geodesic ending at $w\in I$.
\end{proposition}

\begin{lemma}[{\cite[pp. 214-215]{Matsumoto}}]
The subset $\Delta\subset \Gamma$ constituted by the transformations $T_{ji}$ for every Fried dynamics such that $T_{ji}$ has almost no rotation generates a discrete subgroup of $G_1$.
\end{lemma}
\begin{proof}
If $D\colon\widetilde M \to \R^n-I$ is a cover onto a simply connected complement, it is certainly true since the whole holonomy group $\Gamma\supset \langle \Delta\rangle$ would be discrete.
We need to treat the case where $I$ has codimension $2$.

Consider $G_1'$ the stabilizer of $I$. Assume for simplicity that $0\in I$. Note that $I$ must be stable under $A\subset G_1$. Then $\Gamma\subset G_1'$. Consider also the cover $Q\times \R \to \R^n-I$ consisting in taking a  half-hyperplane $Q$ (stable by $A\subset G_1'$) with $I$ as boundary and rotating it around $I$. The transformation group of this cover is  $G_1'\times \R$. The developing map $D\colon\widetilde M \to \R^n-I$ (that is a cover) is lifted to a diffeomorphism $D\colon\widetilde M \to Q\times \R$. In $Q\times \R$ the new holonomy group $\widetilde \Gamma$ is discrete. Now, if $T_{ji}$ has almost no rotation, then $T_{ji}$ preserves almost $P$ and therefore $\widetilde{T_{ji}}\simeq T_{ji}\times \{0\}$, showing that $\langle \Delta\rangle \simeq \langle \widetilde \Delta\rangle$ is discrete.
\end{proof}

\begin{proof}[Proof of the proposition]
The hypothesis $F=\{0\}$ implies that an asymptotic fixed point $q_i$ for $T_{ji}$ is necessarily equal to $D(\gamma)(1)$ since $q_i\in F|_{D(\gamma)(1)} = \{D(\gamma)(1)\}$.

To prove the proposition, we show that $I\cap E|_{D(\gamma)(1)}$ is always reduced to $D(\gamma)(1)$. It will imply $I = P|_{D(\gamma)(1)}$.

Assume that $T_{ji}$ are the transformation of the Fried dynamics based at the initial $D(\gamma)(1)$. Let $R$ be any other transformation $T_{mn}$ but for a different Fried dynamics associated to a geodesic with its endpoint in $E|_{D(\gamma)(1)}$ but different from $D(\gamma)(1)$. (It can be chosen so if, and only if, $I\neq P|_{D(\gamma)(1)}$.)

By the preceding lemma the subgroup $\langle R,\{T_{ji}\}\rangle$ must be discrete. Now consider
\begin{equation}
G_{j} = T_{ji}RT_{ji}^{-1}
\end{equation}
for an $i>0$ fixed and large enough. We show that $G_j$ must converge without being constant, a contradiction. 

Assume for simplicity that $D(\gamma)(1)=0$.
For $i>0$ fixed, write $T_{ji}(x) = c_{ji}+f_{ji}(x)$ with $c_{ji}\to 0$ and write also $R(x) = b + h(x)$. Then the linear part of $G_j$ is given by $f_{ji}hf_{ji}^{-1}$ and therefore must converge in $KA\subset G_1$.

Now, the translational part is given by $c_{ji}+f_{ji}(b) - f_{ji}hf_{ji}^{-1}(c_{ji})$. Each term must converge and therefore so does the sum. 

Therefore $G_j$ converge, but cannot be constant. Indeed, otherwise, the translational part converge to $f_{ji}(b)$ (since $c_{ji}\to 0$) and should be constant. But $f_{ji}(b)$ cannot be constant since $b\neq 0$ and $b\in E$, since $R$ is associated to a geodesic with endpoint different from $D(\gamma)(1)$ but in $E$.
\end{proof}

\paragraph{Example 1}Fried theorem~\cite{Fried} and other generalizations on $\R^n$~\cite{Ale}, all depend dynamically on the hypothesis $F=P=\{0\}$. The preceding proposition shows that an incomplete manifold is radiant.

\paragraph{Example 2}We can easily construct examples with $F=\{0\}$ and $P\neq \{0\}$. Let $K=\{e\}$ and $A = \{(\lambda x, y )\}$ for $\lambda >0$, $x\in \R^k$ and $y\in \R^m$. Then let $\Gamma$ be a subgroup generated by $(2x,y)$ and a lattice on $\R^m\ni y$. Then $\R^k-\{0\} \times \R^m$ quotiented by this subgroup gives a product of a radiant manifold with a Euclidean manifold. Here, $I = \{0\}\times \R^m$.

\paragraph{Note}In the case of Carrière~\cite{Carriere}, the hypothesis of $1$-discompacity implies necessarily $I = (P\oplus F)|_{D(\gamma)(1)}$ by dimensionality. Now the same construction given by the last proposition can be applied by taking $T_{ji}^{-1}HT_{ji}$: it furnishes a convergent  sequence of transformations (a contradiction). Therefore (for such ray geometries) it does not require an additional argument such that the irreducibility given by Goldman-Hirsch~\cite{GH}.

\paragraph{Open question}Is it true that for any rank one ray geometry we always have $I= P|_{D(\gamma)(1)}$? It would show the completeness of the structures having a non-zero $F$ for any Fried dynamics.  In the next section, we  show the completeness for the structures having a parallel volume (for those, $F$ can never be zero), but we believe that this dynamic-geometric property on $I$ is intriguing.

\section{Reducibility of incomplete manifolds}\label{sec-4}

\paragraph{Markus conjecture}
A first consequence of theorem~\ref{thm-comp} is about  Markus conjecture~\cite{Markus}. This conjectures states that  \emph{closed manifolds with parallel volume are complete}.

The fact that an incomplete manifold has its holonomy that  preserves $I$ implies that the holonomy is reducible. But by Goldman-Hirsch~\cite{GH}, the holonomy of  a closed manifold with parallel volume can never be reducible.

\begin{corollary}\label{thm-markus}
\thmmarkus
\end{corollary}

\paragraph{The automorphism group}

It is a vague conjecture~\cite{Gromov} that geometric manifolds with large automorphism groups  should be classifiable.

\begin{definition}
Let $M$ be a $(G,X)$-manifold. An \emph{automorphism} $f\colon M \to M$ is a diffeomorphism such that if $\widetilde f\colon\widetilde M \to \widetilde M$ is any lift  then there exists a unique $\chi(\widetilde f)\in N(\Gamma)$ in the normalizer of the holonomy group, such that $D(\widetilde f(x)) = \chi(\widetilde f)D(x)$.
\end{definition}

By unicity of $\chi(\widetilde f)$, if $\widetilde f_1$ and $\widetilde f_2$ are two lifts of $f$ then $\widetilde f_2 = g\widetilde f_1$ for $g\in \pi_1(M)$ and therefore $\chi(\widetilde f_2) = \rho(g)\chi(\widetilde f_1)$.

The hypothesis that $\chi(\widetilde f)$ normalizes $\Gamma$ follows from the fact that if $\widetilde f$ lifts $f$ then it must preserves the fibers $\pi_1(M)$ of $\widetilde M$ and therefore 
\begin{equation}
\chi(\widetilde f) \Gamma \cdot D(x) = D(\widetilde f(\pi_1(M) \cdot x)) = D(\pi_1(M)\cdot \widetilde f(x)) = \Gamma\cdot  \chi(\widetilde f)D(x).
\end{equation}

\begin{theorem}\label{thm-auto}
\thmauto
\end{theorem}

The fact that $\Aut(M)$ acts non properly on $M$ is equivalent to the existence of  $x_n\to x$ in $M$ and diffeomorphisms $f_n$ such that $f_n(x_n)\to y$, and such that in $\Gamma\backslash N(\Gamma)$, lifts of $f_n$ escape every compact. So we can assume that $x_n\to x$ in $\widetilde M$, $g_n\in N(\Gamma)$ with $\Gamma g_n$ escaping every compact of $\Gamma\backslash N(\Gamma)$, and $y\in\widetilde M$ such that $\Gamma g_n  D(x_n) \to \Gamma D(y)$. 

\begin{proof}
Assume that $M$ is not complete, then by theorem~\ref{thm-comp}, we get that $D\colon\widetilde M \to \R^n-I$ is a cover. Therefore the holonomy $\Gamma$ preserves $I$ and so does the normalizer $N(\Gamma)$. Both must be subgroups of $I\rtimes KA_1$.
We show that $\Aut(M)$ must act properly.

Choose a base point $p\in I$ which is the asymptotic fixed point of a Fried dynamics and consider the decomposition $I\oplus V=\R^n$ with $V\subset E$.

\vskip10pt
Assume that $\Aut(M)$ acts non properly.  As noted, let $g_n\in N(\Gamma)$ be escaping every compact of $\Gamma \backslash N(\Gamma)$, let $x_n\to x_\infty$ and $y$ such that $\Gamma g_nD(x_n)\to \Gamma D(y)$.
 Write
$g_n(x) = c_n + f_n(x)$ with $c_n\in I$ and $f_n\in KA_1$.
The Fried dynamics $T_{ji}\in \Gamma$ considered can be written $T_{ji}{x} = c_{ji}+f_{ji}(x)$. Note that
\begin{equation}
T_{ji}g_n(x) = c_{ji} + f_{ji}(c_n) + f_{ji}f_n(x).
\end{equation}
It shows that we can assume $f_{ji}f_n$ bounded in $KA_1$ since $A_1$ has rank one (up to exchange $T_{ji}$ with $T_{ji}^{-1}$).

Hence there exists  $\gamma_n\in \Gamma$ such that $\gamma_n g_n=  h_n$ has its $KA_1$-factor that converges, up to  a subsequence. The convergence  $\Gamma g_n D(x_n)\to \Gamma D(y)$ says that there exists  $\eta_n\in \Gamma$ such that $\eta_nh_n D(x_n) \to D(y)$.
 
Write 
$\eta_n(x) = b_n + q_n(x)$ and $h_n(x) = c_n + r_n(x)$ we have by construction $r_n\to r$ and
 \begin{equation}
 (\eta_n h_n)(D(x_n)) = b_n + q_n(c_n) + q_nr_n(D(x_n)) \to D(y).
 \end{equation}
 We again have a decomposition in $I\oplus V = \R^n$.
 The term $b_n+q_n(c_n)$ must belong to $I$ since $c_n\in I$. So the $V$-coordinate $((\eta_nh_n)(D(x_n)))_V$ tending to  $D(y)_V$ is determined by $(q_nr_n(D(x_n)))_V$.
But since $D(y)_V\neq 0$ and $V\subset E$, it shows that $q_n$ itself must converge to $q\in KA_1$ (note that $r_n(D(x_n))\to r(D(x_\infty))$).

The $I$-coordinate $((\eta_nh_n)(D(x_n)))_I$ tending to $D(y)_I$ has the same limit as the $I$-factor of $b_n+q_n(c_n)+qr(D(x_n))$.
 Since $D(y)_I$ and $D(x_\infty)$ are both finite, $b_n+q(c_n)$ must converge.
 
 Therefore $\eta_nh_n$ converges in $I\rtimes KA_1$, contradicting the fact that $\Gamma g_n$ escapes every compact of $\Gamma\backslash N(\Gamma)$.
\end{proof}

\paragraph{In higher rank}We give a relatively generic example showing that in higher rank  this phenomenon can  no longer be true. Consider the rank two ray geometry given by the diagonal action of  $(\beta_1 x, \beta_1\beta_2 y)$. Then consider the radiant manifold  $\R^2-\{0\} / \langle (2x,2y)\rangle$. Let $f(x,y) = (x/2,y)$ and $p=(x,1)$. Then $f^n(p)\to (0,1)$ and it corresponds in $M$ to an automorphism acting non properly. The radiant manifold being incomplete, it gives a counter-example in rank two.

\printbibliography
\end{document}